\documentclass[12pt]{amsart}

\usepackage{cancel}
\usepackage{latexsym, amssymb, amsmath,esint}
\usepackage{soul}
\usepackage{amsfonts, graphicx}
\usepackage{graphicx,color}
\newcommand{\be}{\begin{equation}}
\newcommand{\ee}{\end{equation}}
\newcommand{\beq}{\begin{eqnarray}}
\newcommand{\eeq}{\end{eqnarray}}

\usepackage{wasysym,stmaryrd}
\newtheorem{thm}{Theorem}[section]

\newtheorem{lma}[thm]{Lemma}
\newtheorem{prop}[thm]{Proposition}
\newtheorem{cor}[thm]{Corollary}
\newtheorem{defn}[thm]{Definition}
\theoremstyle{remark}
\newtheorem{rem}[thm]{Remark}
\numberwithin{equation}{section}

\def\be{\begin{equation}}
\def\ee{\end{equation}}
\def\bee{\begin{equation*}}
\def\eee{\end{equation*}}
\def\ol{\overline}
\def\lf{\left}
\def\ri{\right}

\def\K{K\"ahler }

\def\KR{K\"ahler-Ricci }
\def\Ric{\text{\rm Ric}}
\def\Rm{\text{\rm Rm}}

\def\ddb{\sqrt{-1}\partial\bar\partial}
\def\p{\partial}

\def\aint{\frac{\ \ }{\ \ }{\hskip -0.4cm}\int}
\def\ol{\overline}
\def\heat{\lf(\frac{\p}{\p t}-\Delta_{g(t)}\ri)}
\def\tr{\operatorname{tr}}

\def\e{\varepsilon}

\def\a{{\alpha}}
\def\b{{\beta}}

\def\R{\mathbb{R}}
\def\C{\mathbb{C}}

\begin{document}

\title[]
{Expanding Ricci solitons coming out of weakly PIC1 metric cones}

\author{Pak-Yeung Chan}
\address[Pak-Yeung Chan]{Mathematics Institute, Zeeman Building, University of Warwick, Coventry CV4 7AL, UK}
\email{pak-yeung.chan@warwick.ac.uk}

 \author{Man-Chun Lee}
\address[Man-Chun Lee]{Department of Mathematics, The Chinese University of Hong Kong, Shatin, Hong Kong, China}
\email{mclee@math.cuhk.edu.hk}

 \author{Luke T. Peachey}
\address[Luke T. Peachey]{Department of Mathematics, The Chinese University of Hong Kong, Shatin, Hong Kong, China}
\email{lukepeachey@cuhk.edu.hk}


\renewcommand{\subjclassname}{
  \textup{2020} Mathematics Subject Classification}

\date{\today}

\begin{abstract}
Motivated by recent work of Deruelle-Schulze-Simon \cite{DSS2024}, we study complete weakly PIC1 Ricci flows with Euclidean volume growth coming out of metric cones. We show that such a Ricci flow must be an expanding gradient Ricci soliton, and as a consequence, any metric cone at infinity of a complete weakly PIC1 K\"ahler manifold with Euclidean volume growth is biholomorphic to complex Euclidean space in a canonical way.
\end{abstract}

\maketitle

\markboth{Pak-Yeung Chan, Man-Chun Lee, Luke T. Peachey}{Expanding Ricci solitons coming out of weakly PIC1 metric cones}

\section{Introduction}\label{s-introduction}

Metric cones arise naturally as tangent cones of Gromov-Hausdorff limits of sequences of non-collapsing smooth manifolds with uniform Ricci curvature lower bounds due to the celebrated work of Cheeger-Colding \cite{CheegerColding}. 
In this paper, we are interested in studying Ricci flows coming out of metric cones in the following sense:

\begin{defn}
Given a Ricci flow $g(t)$ defined on $M$ for $t \in (0,T)$, we say that $g(t)$ is coming out of a metric cone $C(X)$ if
$$(M,d_{g(t)},x_0) \xrightarrow{t \downarrow 0} (C(X),d_c,o),$$
in the pointed Gromov-Hausdorff sense, where $d_c$ denotes the conical metric, and $o$ the vertex of the cone, for some $x_0 \in M$.
\end{defn}

Since a metric cone is self-similar under dilation, it follows heuristically that a Ricci flow coming out of a metric cone should itself be self-similar. More precisely, given a complete Riemannian manifold $(M,g)$ and a function $f\in C^{\infty}(M)$, the triple $(M,g,\nabla f)$ is said to be an \textit{expanding gradient Ricci soliton} if
\begin{equation}\label{RSe}
2\nabla^{g,2} f - g -2\Ric(g) = 0.  
\end{equation}
An expanding gradient Ricci soliton generates an immortal self-similar solution to the Ricci flow modulo scaling and re-parametrization: since the vector field $\nabla^g f$ is complete \cite{Zhang}, if $\{\phi_t: M \rightarrow M\}_{t>0}$ denotes the flow of the time-dependent vector field $-t^{-1}\nabla^g f$ with $\phi_1 = id_M$, then the family of self-similar metrics $g(t):=t \phi_t^*g$ satisfy the Ricci flow equation. 

Expanding Ricci solitons are a generalisation of Einstein manifolds with negative Einstein constant. Recently, Topping and the third named author completely classified expanding Ricci solitons in dimension two \cite{PeacheyTopping23}.  Since \textit{gradient} expanding Ricci solitons are deemed to be the prototype for Ricci flows coming out of conical singularities (see for instance \cite{GS2018, Lav2023}),  we restrict our attention in this paper to those solitons with soliton vector field the gradient of a smooth function as in \eqref{RSe}.

The resolution of conical singularities by the Ricci flow is related to the singularity analysis of the flow. To formulate a reasonable weak solution to the Ricci flow in higher dimensions, it is of fundamental importance to investigate the properties of Ricci flow through singularities. Indeed, it has been proposed that asymptotically conical expanding solitons could be used to resolve the conical  singularities that appear in 4-dimensional Ricci flows \cite{GS2018,AD2024, BamlerChen2023}.


In recent years, substantial progresses have been made to understand the resolution of conical singularities. Schulze and Simon \cite{SchulzeSimon2013} proved that any tangent cone at infinity of a complete manifold with maximal volume growth and non-negative curvature operator can be smoothed out by an expanding gradient Ricci soliton. The problem has also been extensively studied when the link of the cone $X$ is a smooth closed manifold. Bryant applied ODE methods to construct expanding gradient Ricci solitons on $\R^n$ coming out of rotationally symmetric cones (see also \cite{FIK2023}). When the link $X$ is a simply connected smooth closed manifold with $\Rm(g_X)\ge 1$, Deruelle \cite{Deruelle2016} used a PDE deformation method to construct expanding gradient Ricci solitons with non-negative curvature operator coming out of $C(X)$. Moreover, a uniqueness result was proven within the class of expanding gradient Ricci solitons with $\Rm>0$ asymptotic to $C(X)$ (see also \cite{Siepmann2013, Chodosh2014, ChodoshFong2016, CD2020, CDS2019}). The existence result has recently been generalized in dimension $4$ by Bamler and Chen \cite{BamlerChen2023} using a new degree theory method which only requires $X$ to be a smooth quotient of $\mathbb{S}^3$ with a metric of scalar curvature $\mathcal{R}(g_X)\ge 6$. Angenent and Knopf \cite{AD2024} constructed explicit examples of metric cones in dimensions five and higher which can be smoothed out by an arbitrary finite number of geometrically distinct expanding gradient Ricci solitons. Thus the unique behavior of the Ricci flow through a singularity is not expected in higher dimensions.

In this paper we consider those metric cones which arise as a tangent cone at infinity of a non-collapsed weakly PIC 1 manifold. That is, metric cones $C(X)$ such that $\left(C(X),d_c,o\right)$ is the pointed Gromov-Hausdorff limit of $(M,R_i^{-2}g,x_0)$ for some $R_i\to +\infty$, where $(M,g)$ is a complete manifold with non-negative $1$-isotropic curvature and Euclidean volume growth. Isotropic curvature was first introduced by Micallef and Moore \cite{MicallefMoore1988} to prove the topological sphere theorem, and is one of the most natural curvature quantities to consider along the Ricci flow, its study along the flow dating back to Hamilton \cite{Hamilton1997}. In the compact case, substantial progress in understanding positive isotropic curvature has been made by recently Brendle \cite{BrendleISO}. Some important applications of the Ricci flow theory of isotropic curvature also include the differentiable sphere theorem \cite{BrendleSchoen2008, BrendleSchoen2009}. We refer interested readers to \cite{BrendleBook,PIC1} for an overview on related topics. Let us first recall the notion of PIC1 curvature.


\begin{defn}
An algebraic curvature tensor $R \in Sym^2(\Lambda^2(\mathbb{R}^n))$ is weakly PIC1, denoted by $R \in \mathrm{C}_{PIC1}$, if its $\mathbb{C}$-linear extension
$R \in Sym^2(\Lambda^2(\mathbb{C}^n))$ satisfies $R(\omega,\ol{\omega})\geq 0$ for all simple and isotropic $\omega \in \Lambda^2(\mathbb{C}^n)$.
\end{defn}
Recall, a complex plane $\omega \in \Lambda^2(\mathbb{C}^n)$ is \textit{isotropic} if $\mathcal{I}(\omega,\omega) = 0$, where $\mathcal{I}$ denotes the $\mathbb{C}$-linear extension of the curvature operator with constant sectional curvature $1$ (see \cite{PIC1} for more details). We say that a Riemannian manifold $(M,g)$ has non-negative $1$-isotropic curvature if $\mathrm{Rm}(g)\in  \mathrm{C}_{PIC1}$ on $M$. 
In dimension $3$, non-negative Ricci curvature is equivalent to negative $1$-isotropic curvature. When dimension $\ge 4$, $\Ric(g) \ge 0$ if $\mathrm{Rm}(g)\in  \mathrm{C}_{PIC1}$.

We now state the main result of our paper: any metric cone at infinity of a non-collapsed wealky PIC1 manifold is resolved by an expanding gradient Ricci soliton.

\begin{thm}\label{thm:main-1}
Suppose $(M^n,g_0)$ is a complete non-compact manifold such that $\mathrm{Rm}(g_0)\in \mathrm{C}_{PIC1}$ and $\mathrm{AVR}(g_0)>0$. For any metric cone at infinity $(C(X),d_c,o)$ of $(M,g_0)$, there exists a smooth manifold $M_\infty$, a metric $d_\infty$ on $M_\infty$, a smooth Ricci flow $g(t)$ on $M_\infty$ for $t>0$, and a smooth function $u$ on $M_\infty\times (0,+\infty)$, such that 
\begin{enumerate}
    \item[(a)] $\mathrm{Rm}(g(t))\in \mathrm{C}_{PIC1}$;
    \item[(b)] \label{curvb} $|\Rm(g(t))|\leq \a t^{-1}$ for some $\a>0$;
     \item[(c)] $(M_\infty,d_\infty,x_\infty)$ is isometric to $(C(X),d_c,o)$ as pointed metric spaces;
    \item[(d)] $d_{g(t)} \to d_\infty$ uniformly on $M_\infty$ as $t\downarrow 0$;
    \item[(e)] $u(\cdot,t)\to \frac14 d_\infty^2(x_\infty,\cdot)$ in $C^\b_{loc}(M_\infty)$ as $t \downarrow 0$, for some $\b\in (0,1)$;
    \item[(f)] $2t\,\Ric(g(t))+g(t)-2\nabla^{2,g(t)}u \equiv 0$, for all $t>0$. 
\end{enumerate}
\end{thm}
A more general result holds as long as the initial data of the Ricci flow is a metric cone in the pointed Gromov-Hausdorff sense and satisfies the PIC1 condition along the flow. We refer readers to Theorem~\ref{thm:Rie-case} for the detailed statement. 

{ 
\begin{rem}\label{rmk ot}
It follows from conditions (b) and (f) that the Ricci flow $g(t)$ is self-similar. Indeed, let $\phi_t$ be the flow of $-\nabla^{g(1)}u(\cdot,1)/t$ with $\phi_1=id_M$. As both $g(t)$ and $t\phi_t^*\left[g(1)\right]$ are complete Ricci flows satisfying condition (b) and agreeing at $t=1$, by the forward and backward uniqueness of the Ricci flow \cite{ChenZhu2006,Kotschwar2010}, $g(t)\equiv t\phi_t^*\left[g(1)\right]$, for all $t>0$.
\end{rem}}

In contrast with earlier work \cite{Siepmann2013, Deruelle2016, CD2020}, we do not require $C(X)$ to have isolated singularities. In a similar spirit to \cite{SchulzeSimon2013} our result doesn't make any regularity assumption on the initial cone and hence works for rougher metric cones. On the other hand, we generalize the curvature condition in the earlier works \cite{ChenZhu2003,Brendle2009,SchulzeSimon2013} where the differential Harnack inequality along the Ricci flow plays a crucial role. Indeed, the importance of relaxing the curvature condition from weakly PIC2 to weakly PIC1 is largely motivated by a conjecture of Hamilton-Lott which states that a complete, connected 3-dimensional Riemannian manifold which is uniformly Ricci pinched is either compact or flat. The conjecture was answered affirmatively in full generality by the work of Lott \cite{Lott2024}, Deruelle-Schulze-Simon \cite{DSS2022} and Lee-Topping \cite{LeeTop2022}. To extend the result to higher dimensions, one asks if a weakly PIC1 Ricci flow coming out of a metric cone in the pointed Gromov-Hausdorff sense is close to an expanding Ricci soliton. In \cite{DSS2024}, Deruelle-Schulze-Simon were able to show that under a Reifenberg regularity assumption, the Ricci flow smoothing is asymptotically similar to an expanding Ricci soliton. Thus, they show that a complete PIC1 pinched manifold with maximal volume growth is flat or compact (see also the work of the second named author and Topping \cite{LeeToppingPIC1}). In this regard, our result says that the Ricci flow smoothing is automatically a gradient expanding Ricci soliton generalizing the work of \cite{DSS2024} in the  weakly PIC1 case. Our method is indeed largely motivated by the key observation in \cite{DSS2024} which considered an obstruction tensor $\mathcal{T}=-t\Ric-\frac12 g+\nabla^2 u$ and an obstruction function $v=|\nabla u|^2-u+t^2\mathcal{R}+2t\tr_g \mathcal{T}$,
where $u$ solves $\heat u=-\frac{n}2$ with initial data $\frac14 d_c^2(\cdot,o)$ locally. Here $\mathcal{R}$ denotes the scalar curvature. By combining Ricci flow smoothing estimates and a sharp modification of $d_c^2$ by Cheeger-Jiang-Naber \cite{CJN21}, they managed to control $\mathcal{T}$ and $v$ locally away from the tips. In this work, we are able to estimate up to the tips. This is based on a heat kernel estimate of Bamler-Cabezas-Wilking \cite{BamlerCabezasWilking2019} and a localized maximum principle developed by the second named author and Tam \cite{LeeTam2022}. With the soliton structure on the metric cone at infinity, we give an alternative proof to the main Theorem in \cite{DSS2024} saying that any complete non-compact manifolds with PIC1 pinched and Euclidean volume growth must be isometric to Euclidean space, see Corollary~\ref{cor:pinching}.

Our second motivation comes from a longstanding open question of Yau in complex geometry which asks if a complete non-compact \K manifold $(M,g)$ with positive bisectional curvature is biholomorphic to $\mathbb{C}^n$. This can be regarded as a non-compact version of the Frankel conjecture. It was discovered by Chau-Tam \cite{ChauTamMRL} that the existence of an expanding \KR soliton metric with $\Ric\geq 0$ is sufficient to conclude $M $ is biholomorphic to $\C^n$. This important observation led them to confirm Yau's uniformization conjecture in the maximal volume growth case with bounded curvature by studying the long-time behaviour of the \KR flow with non-negative bisectional curvature \cite{ChauTam2006JDG}. We therefore consider the \KR flow smoothing of the corresponding metric cone under the weakly PIC1 condition and prove the following:

\begin{cor}\label{cor:kahler}
Let $(M,g)$ be a complete non-compact \K manifold with $\mathrm{dim}_\mathbb{C}(M)=n$, Euclidean volume growth and non-negative $1$-isotropic curvature. If $(M_\infty,d_\infty,x_\infty)$ is a metric cone at infinity of $M$, then the metric cone $(M_\infty,d_\infty,x_\infty)$ is a pointed complete metric K\"ahler space which is biholomorphic to $\mathbb{C}^n$.
\end{cor}
The notion of complete metric K\"ahler space  ensures the compatibility between the holomorphic structure and the metric structure, see \cite[Section 5]{LottDuke}. Indeed, by the \KR flow smoothing, if $(M_\infty,d_\infty,x_\infty)$ is a metric cone arising as a limit of $(M,R_i^{-2}g,x_0)$ for some $R_i\to +\infty$, then the complex structure also converges modulus pulling back by diffeomorphism \cite{LeeTamGT,LottDuke} and hence the complex structure on $M_\infty$ is the most natural one induced from $M$. Part of our result in the K\"ahler case only requires weaker curvature conditions, namely non-negative Ricci curvature and non-negative orthogonal bisectional curvature. Recall, a \K manfiold $(M,g)$ has non-negative orthogonal bisectional curvature, denoted by $\mathrm{OB}(g)\geq 0$, if 
$$R(X,\bar X,Y,\bar Y)\geq 0, \quad \forall X,Y\in T^{1,0}(M) \text{  such that  } g(X,\bar Y)=0.$$
We refer readers to Theorem~\ref{thm:Kahler-OB-Ric} for the detailed statement.

\medskip
The paper is organized as follows. In section $2$ we present some preliminary results on the existence of suitable Ricci flows, (conjugate) heat kernel estimates, distance distortion estimates under the Ricci flow and the local maximum principle from \cite{LeeTam2022}. In section $3$, we establish a parabolic differential inequality satisfied weakly by the eigenvalues of the soliton obstruction tensor under non-negative $1$-isotropic curvature condition, with section $4$ dedicated to its \K analogy. Using the regularization of conical distance functions by Cheeger-Jiang-Naber \cite{CJN21}, in section $5$ we solve the global Cauchy problem for the heat flow starting from the square of the conical distance function. This gives rise to a global soliton obstruction tensor associated to the flow. With these analytic preparation, we apply the local maximum principle to show that this obstruction tensor vanishes everywhere and hence the flow is an expanding gradient soliton.

\vskip0.2cm

{\it Acknowledgement:} The authors would like to thank Alix Deruelle, Felix Schulze and Peter Topping for some insightful discussions.  The first named author is supported by EPSRC grant EP/T019824/1. The second named author was partially supported by Hong Kong RGC grant (Early Career Scheme) of Hong Kong No. 24304222, No. 14300623,  and a NSFC grant No. 12222122.


\section{Preliminaries on Ricci flow smoothing} 
In this section, we will collect some known results for Ricci flows starting from complete non-compact manifolds with possibly unbounded curvature, in both the Riemannian and \K cases. We start with the following Ricci flow existence result.
\begin{thm}\label{thm:Ricciflow-PIC1}
If $(M^n,g_0)$ is a complete non-compact manifold such that 
\begin{enumerate}
    \item[(A)] $\mathrm{Vol}_{g_0}\left(B_{g_0}(x,r) \right)\geq vr^{n}>0$, for all $x\in M$ and $r>0$;
    \item[(B)] either one of the following holds 
    \begin{enumerate}
        \item[(i)] $\mathrm{Rm}(g_0)\in \mathrm{C}_{PIC1}$ or;
        \item[(ii)] $g_0$ is \K with $\mathrm{OB}(g_0),\Ric(g_0)\geq 0$,
    \end{enumerate}
\end{enumerate}
then there exists $\a(n,v)>0$ and an immortal solution $g(t)$ to the Ricci flow on $M\times [0,+\infty)$, such that for all $t>0$, we have 
\begin{enumerate}
 \item[(I)] $\sup_M|\Rm(g(t))|\leq \a t^{-1}$;
 \item[(II)] $\mathrm{Rm}(g(t))\in \mathrm{C}_{PIC1}$ if assumption (B)(i) holds;
 \item[(III)] $g(t)$ is \K with $\mathrm{OB}(g(t)),\Ric(g(t))\geq 0$ if assumption (B)(ii) holds;
 \item[(IV)]$\mathrm{Vol}_{g(t)}\left(B_{g(t)}(x,r) \right)\geq v r^n>0$, for all $x\in M$ and $r>0$;
   
\end{enumerate}
Furthermore, there exists $C_0(n,v)>0$ such that the heat kernel $K(x,t;y,s)$ and the conjugate heat kernel \footnote{More precisely, $K$ denotes the kernel with respect to the operator $\partial_t -\Delta_{g(t)}$, while $G$ denotes the kernel corresponding to $\partial_t -\Delta_{g(t)} -\mathcal{R}$.} $G(x,t;y,s)$ satisfy the inequality
$$K(x,t;y,s)\leq G(x,t;y,s)\leq \frac{C_0}{t^{n/2}}\exp\left(-\frac{d_{g(s)}^2(x,y)}{C_0t} \right),$$
for all $0\leq 2s<t<+\infty$, and $x,y\in M$. 
\end{thm}

\begin{proof}
In the Riemannian setting (B)(i), the short-time existence of a Ricci flow satisfying (I) and (II) was proven by Simon-Topping \cite{SimonToppingGT} when $n=3$, and Lai \cite{Lai2019} when $n \geq 4$. In the \K setting (B)(ii), the short-time existence of a Ricci flow with properties (I) and (III) was proven by the second named author and Tam \cite{LeeTamGT}. The long-time existence of such a Ricci flow then follows from a scaling argument (see for example \cite[Theorem 1.1 \& Corollary 4.1]{HeLee2021}). More precisely, since the assumptions are scaling invariant, by rescaling the metric, running the Ricci flow, and then parabolically rescaling the flow back, we find a sequence of Ricci flows which exist for an arbitrarily long time. Using (I), the long-time existence follows after passing to a subsequence and taking a limit. (IV) follows from the preservation of the asymptotic volume ratio under $\Ric\geq 0$ and $|\Rm|\leq \a t^{-1}$ \cite{Yok}. The Gaussian upper bound on the conjugate heat kernel $G$ follows from  Bamler, Cabezas-Rivas, and Wilking \cite[Proposition 3.1]{BamlerCabezasWilking2019} and a scaling argument (see also \cite{LeeTam2022,ChanLee2023}). It remains to estimate the heat kernel $K$. Since the scalar curvature $\mathcal{R}\geq 0$ along the Ricci flow, for all $0<s<t<+\infty$
\begin{equation}
\begin{split}
    \heat G(x,t;y,s)&=\mathcal{R}_{g(t)}G(x,t;y,s)\\
   & \geq 0=\heat K(x,t;y,s).
\end{split}
\end{equation}

Since $g(t)$ has bounded curvature for $t>0$ and $$\lim_{t\to s}G(x,t;y,s)=\lim_{t\to s}K(x,t;y,s)=\delta_y(x),$$ we deduce from the standard maximum principle 
that $G(x,t;y,s)\geq K(x,t;y,s)$ for all $t>s>0$. The assertion when $s=0$ follows by taking $s\downarrow 0$.  
\end{proof}

We also need the distance distortion estimate along Ricci flow originally due to Simon-Topping \cite{SimonToppingJDG,SimonToppingGT}, see also \cite{Hamilton-Formation}.
\begin{prop}\label{prop:dist}
Suppose $(M^n,g(t)),t\in (0,T]$ is a smooth complete Ricci flow such that 
\begin{equation}
    \Ric(g(t))\geq 0,\;\; |\Rm(g(t))|\leq \a t^{-1},
\end{equation}
for some $\a>0$. Then the following three statements hold:
\begin{enumerate}
    \item[(i)] there exists $C_0(n,\a)>0$ such that for all $0<s\leq t\leq T$ and $x,y\in M$, 
    $$d_{g(t)}(x,y)\leq d_{g(s)}(x,y)\leq d_{g(t)}(x,y)+C_0\sqrt{t-s};$$
    \item[(ii)] there is a well-defined metric $d_0$ on $M\times M$ as $t\downarrow 0$,
    $$d_0(x,y) := \lim_{t\to 0}d_{g(t)}(x,y);$$
   
   \item[(iii)] there exists $\gamma(n,\a,T)>0$ such that for all $x,y\in M\times (0,T]$, $$d_{g(t)}(x,y)\geq \gamma \left( d_0(x,y) \right)^{1+2n\a}.$$   
In particular,  the metric $d_0$ generates the same topology as that of $(M,d_{g(t)})$ for all $t\in (0,T]$.
\end{enumerate}
\end{prop}
\medskip

Next, we need a local maximum principle proven in \cite[Theorem 1.1]{{LeeTam2022}} which plays key role in localizing the error.
\begin{thm}\label{thm:Local-MP}
    Let $(M^n, g(t))$, $t\in [0,T]$ be a smooth (not necessarily complete) solution to the Ricci flow such that \[
    \Ric(g(t))\leq \alpha t^{-1}
    \]
    for $t\in (0,T]$. Suppose that $\varphi(x,t)$ and $L(x,t)$ are continuous functions on $M\times [0,T]$ such that 
    \[
    \varphi(x,t)\leq \alpha t^{-1}, \quad L(x,t)\leq \alpha t^{-1},
    \] 
     and 
    \[
    \left(\frac{\partial}{\partial t}-\Delta_{g(t)}\right)\varphi\Big\vert_{(x_0,t_0)}\leq L(x_0,t_0) \varphi(x_0,t_0),
    \]
    in the barrier sense, whenever $\varphi(x_0,t_0)>0$. Suppose $\varphi(\cdot,0)\leq 0$ on $B_{g(0)}(p,2)\Subset M$ for some $p\in M$. Then for any $\ell>\alpha+1$, there exists $\hat T=\hat T(n,\alpha, \ell) > 0$ such that 
    \[
    \varphi(p,t)\leq t^\ell, \quad \forall t\in [0,T\wedge \hat T].
    \]
\end{thm}

\section{Soliton obstruction tensor: Riemannian case}
In this section we follow \cite{DSS2024} and consider the following family of covariant symmetric $2$-tensors
\begin{equation}\label{obten}
    S(t) := -2t\Ric(g(t))-g(t)+2\nabla^2 u, \quad \forall t>0,
\end{equation}
where $g(t)$ is a fixed underlying Ricci flow on $M$ for $t>0$, and $u$ a fixed smooth solution to the heat equation $\heat u =0$ on $M \times (0,\infty)$. Since we only study the rigidity at the level of the Hessian of $u$, we work on the heat equation directly, which differs slightly from \cite{DSS2024}.

Ultimately our goal is to show that, for a Ricci flow as in the setting of Theorem \ref{thm:main-1}, $S\equiv 0$ if the initial data of $u$ is chosen suitably. As noted in Remark \ref{rmk ot}, we may then conclude that our Ricci flow is an expanding gradient Ricci soliton. In order to do this, instead of considering the tensor $S$ or the scalar quantity $|S|^2$ directly, we consider the curvature type tensor $S\owedge g$. This is motivated by the work of Tsui-Wang \cite{TsuiWang} on area decreasing maps between manifolds evolving under mean curvature flow. The following is a simple linear algebra fact.  
\begin{lma}\label{lma:linear-alge}
If $n\geq 3$, then $S\owedge g\equiv 0$ is equivalent to $S\equiv 0$.
\end{lma}
\begin{proof}
    Since $S$ is symmetric, we might find an orthonormal frame $\{e_i\}_{i=1}^n$ with respect to $g$ such that $S_{ij}=\lambda_i \delta_{ij}$. Hence for $i\neq j$,
    \begin{equation}
        (S\owedge g)_{ijji}=S_{ii} g_{jj}+S_{jj}g_{ii}=\lambda_i+\lambda_j=0.
    \end{equation} 
    Since $n\geq 3$, 
    \begin{equation}
       2\lambda_1=(\lambda_1+\lambda_2)+(\lambda_2+\lambda_3)+(\lambda_1+\lambda_3)-2(\lambda_2+\lambda_3)=0.
    \end{equation}
    Repeating the argument for each $i$, we see that $S\equiv 0$.
\end{proof}

\medskip

For each $(x,t)\in M\times (0,\infty)$, define 
\begin{equation}\label{defn:min-max-S2}
\left\{
\begin{array}{ll}
 \varphi_{-}(x,t)=\displaystyle\inf\left\{\kappa\geq 0:  S\owedge g+\kappa \cdot \frac12 g\owedge g > 0\right\};\\[3mm]
    \varphi_{+}(x,t)=\displaystyle\inf\left\{\kappa\geq 0:  S\owedge g-\kappa \cdot \frac12 g\owedge g < 0\right\},
\end{array}
   \right.
\end{equation}
where the positivity of the $(4,0)$-tensor corresponds to lying within the cone of positive curvature operators. Clearly, if $\varphi_-=\varphi_+=0$ then $S\owedge g\equiv 0$. The following key observation is that under $\mathrm{Rm}(g(t))\in\mathrm{C}_{PIC1}$, both $\varphi_{-}$ and $\varphi_{+}$ satisfy a better evolution inequality. 
\begin{lma}\label{lma:evo-S2}
Given the above setting, suppose $\mathrm{Rm}(g(t))\in \mathrm{C}_{PIC1}$ for all $t>0$. Then, if $\varphi$ is equal to either $\varphi_-$ or $\varphi_+$ as defined in \eqref{defn:min-max-S2},
$$\heat \varphi\leq \mathcal{R}\varphi $$
on $M\times (0,\infty)$ whenever $\varphi\geq 0$ in the barrier sense. Here $\mathcal{R}$ denotes the scalar curvature of $g(t)$.
\end{lma}
\begin{proof}
We only consider $\varphi=\varphi_-$ since the discussion of $\varphi_{+}$ is identical. Recall from \cite[Lemma 2.33 \& Lemma 2.40]{ChowBook} that the tensor $S$ satisfies 
\begin{equation}\label{evo-S-eqn}
    \begin{split}
        \heat S_{ij}&=-2t\left(2R_{iklj} R_{kl}- R_{ik}R_{kj} -R_{jk}R_{ik}\right)\\
        &\quad +2\left(2R_{kijl}u_{kl}-R_{ik}u_{kj}-R_{kj}u_{ik} \right)\\
        &=2R_{iklj}S_{kl}-R_{ik}S_{kj}-R_{kj}S_{ik}.
    \end{split}
\end{equation}
under an orthonormal frame. At a point $(x_0,t_0)$, we may assume that $\varphi \geq  0$ and $S\owedge g+\varphi \cdot \frac12 g\owedge g\in \partial \mathrm{C}_{Rm}$, as otherwise $\varphi\equiv 0$ locally, and the conclusion holds trivially. By compactness of the sphere bundle and the definition of $S\owedge g$, we may find an orthonormal frame $\{e_i\}_{i=1}^n$ such that it diagonalises $S$ at $(x_0,t_0)$ so that $S(e_i,e_j)=\lambda_i \delta_{ij}$ with $\lambda_1\leq \lambda_2...\leq \lambda_n$, and thus
$$\left(S\owedge g\right)|_{(x_0,t_0)}(e_1,e_2,e_2,e_1)=-\varphi(x_0,t_0).$$
By using the Uhlenbeck trick, we extend $\{e_i\}_{i=1}^n$ smoothly to an orthonormal frame $\{E_i\}_{i=1}^n$ around $(x_0,t_0)$ so that $\{E_A\}_{A=1}^n$ satisfies 
\begin{equation}\label{vector-bundle}
    \partial_t E_A^i =R^i_j E^j_A
\end{equation}
for each $A\in \{1,...,n\}$. Hence the function $\ell(x,t)= -(S\owedge g)(E_1,E_2,E_2,E_1)$ satisfies $\ell\leq \varphi$ and $\ell(x_0,t_0)=\varphi(x_0,t_0)=-S_{11}-S_{22}$. Combining \eqref{evo-S-eqn} with \eqref{vector-bundle}, we have
\begin{equation}
    \begin{split}
        \heat \ell(x,t)\Big|_{(x_0,t_0)}
        &=-2\sum_{i=3}^n  \left(R_{1ii1}+R_{2ii2}\right)S_{ii}-2R_{1221}\left(S_{11}+S_{22}\right)\\
        &\leq -\sum_{i=3}^n  \left(R_{1ii1}+R_{2ii2}\right)(S_{11}+S_{22})-2R_{1221}\left(S_{11}+S_{22}\right)\\
        &=\varphi\cdot \left(R_{11}+R_{22}\right)\\
        &\leq \mathcal{R}\ell.
    \end{split}
\end{equation}
Here we have used $R \in \mathrm{C}_{PIC1}$ in the first inequality and $\Ric,\varphi\geq 0$ in the second inequality. This completes the proof. 
\end{proof}

\section{Soliton obstruction tensor: \K case}
In the K\"ahler case, we follow the standard conventions from complex geometry. We say that $g(t)$ is a \KR flow if its induced \K form $\omega(t)$ satisfies
\begin{equation}
    \frac{\partial}{\partial t} \omega(t)=-\Ric(\omega(t))=\ddb \log \omega^n.
\end{equation}

We also follow the complex conventions  $$\mathcal{R}_g=\tr_\omega \Ric, \quad \textnormal{ and }\quad \Delta_{g}f=\tr_\omega \ddb f, \quad \forall f\in C^\infty(M).$$ 

In the \K setting, we shall only consider the $(1,1)$ part of the obstruction tensor, or equivalently, the induced $(1,1)$-forms:
\begin{equation}
\a(t) := -t\, \Ric(\omega(t))-\omega(t) +\ddb u \in \Lambda^{1,1}(M),
\end{equation}
where $\heat u=0$ and $\Ric(\omega(t))=-\ddb\log \omega(t)^n$ is the Ricci form of $\omega(t)$, for all $t>0$. As in the Riemannian case, for each $(x,t) \in M \times (0,\infty)$ we consider similar (but simpler) eigenvalues:
\begin{equation}\label{defn:min-max-K}
\left\{
\begin{array}{ll}
 \varphi_{-}(x,t)=\displaystyle\inf\left\{\kappa \geq 0 :  \alpha + \kappa \omega > 0 \right\};\\[3mm]
    \varphi_{+}(x,t)= \displaystyle\inf\left\{\kappa \geq 0 :   \alpha -\kappa \omega < 0 \right\}
\end{array}
   \right.
\end{equation}

The following lemma is a \K parallel to Lemma~\ref{lma:evo-S2}. 
\begin{lma}\label{lma:evo-S-Kahler}
Given the above setting, suppose our solution to the \K Ricci flow satisfies $\mathrm{OB}(g(t)),\Ric(g(t))\geq 0$ for all $t>0$. Then, if $\varphi$ is equal to either $\varphi_-$ or $\varphi_+$ as defined in \eqref{defn:min-max-K}, 
$$\heat \varphi\leq \mathcal{R}\varphi $$
on $M\times (0,\infty)$ whenever $\varphi\geq 0$ in the barrier sense.
\end{lma}
\begin{proof}
We follow the exact same argument as Lemma~\ref{lma:evo-S2}. Without loss of generality we consider $\varphi=\varphi_-$. Suppose $\varphi(x_0,t_0)> 0$. We may assume that the $(1,1)$-form $\alpha + \varphi \omega$ has non-trivial kernel at $(x_0,t_0)$. Choosing a locally parallel unitary frame $\{E_i\}_{i=1}^n$ in spacetime around the point $(x_0,t_0)$, such that at $(x_0,t_0)$ it diagonalises $\alpha$, and so that $\alpha(E_1 , \ol{E_1}) = -\varphi(x_0,t_0)$, we can define the function $\ell(x,t) := -\alpha(E_1,\ol{E_1})$ which is a lower barrier to $\varphi$ at $(x_0,t_0)$. Using \cite[Lemma 2.33]{ChowBook} again, we have the desired differential inequality
\begin{equation}
\begin{split}
\heat \ell(x,t)\Big|_{(x_0,t_0)} &=-R_{1 \ol{1} p \ol{p}} \alpha_{p \ol{p}}\\
&\leq -\sum_{p=1}^n R_{1 \ol{1} p \ol{p}} \alpha_{1 \ol{1}}\\
&=R_{1\bar 1} \varphi\leq \mathcal{R}\ell.
\end{split}
\end{equation} 
Here we used that $\a_{1\bar 1}$ is the lowest eigenvalue of $\a$ at $(x_0,t_0)$ in the first inequality, and then that $\mathrm{OB},\Ric\geq 0$. 
\end{proof}

\section{Global solution to the heat equation from the square distance function}
In this section we consider those Ricci flows which start from a metric cone. In contrast to \cite{DSS2024}, we will assume stronger curvature conditions on our flow. Let $(M^n, g(t)), t\in (0,\infty)$ be a smooth complete solution to the Ricci flow such that for some $\a,v_0>0$, the following conditions hold for all $t>0$
\begin{enumerate}
\item[(a)] $|\Rm|\leq \a t^{-1}$; 
\item[(b)] either 
\begin{enumerate}
    \item[(i)] $\Rm(g(t))\in \mathrm{C}_{PIC1}$ or,
    \item[(ii)] $g(t)$ is K\"ahler with $\mathrm{OB}(g(t)),\Ric(g(t))\geq 0$;
\end{enumerate}
\item[(c)] $\mathrm{AVR}(g(t))\geq v_0$.
\end{enumerate} 
By Theorem \ref{thm:Ricciflow-PIC1} we may assume the (conjugate) heat kernel estimates for $g(t)$. We also assume that $g(t)$ comes out of a metric cone $C(X)$ in the pointed Gromov-Hausdorff sense:
\[
(M,d_{g(t)},x_0)\xrightarrow[t \downarrow 0]{\text{pGH}} (C(X),d_c,o),
\]
where $o$ is the vertex of the cone. Following \cite{DSS2024}, we want to consider the heat flow from the square of the distance function on $C(X)$ and use it to construct an obstruction tensor $S$ (see \eqref{obten}) which will ultimately measure the discrepancy between $g(t)$ and a suitable expanding Ricci soliton flow. Thanks to Proposition~\ref{prop:dist} part (ii), we may identify $\left(C(X),d_c,o \right)$ with the pointed metric space $(M,d_0,x_0)$, where $d_0$ is the pointwise limit of the metrics $d_{g(t)}$ as $t \downarrow 0$. In this way, we avoid the technicality of using $d_c$ as initial data. Thoughout this section, we will assume this set-up.

\medskip

Since the Ricci flow is not smooth up to $t=0$, we construct our solution to the heat equation more carefully. In particular, we use a regularization result of Cheeger-Jiang-Naber \cite{CJN21} for the distance function at small times. The following is a modified version of \cite[Proposition 3.1]{DSS2024}. 
\begin{lma}\label{lma:appro-initial}
Under the above setting, there exist $t_i,\e_i \downarrow 0$, $R_i:=\e_i^{-\frac1{n}}\to+\infty$ and $u_{i,0}\in C^\infty\left(B_{g(t_i)}(x_0,2R_i)\right)$ such that 
\begin{enumerate}
    \item[(i)] $$\Delta_{g(t_i)}u_{i,0}=\frac{n}2;$$
    \item[(ii)] $$\sup_{B_{g(t_i)}(x_0,2R_i)}\left|u_{i,0}-\frac{1}{4}d_{g(t_i)}^2(x_0,\cdot)\right|\leq \e_i \leq 1$$
    \item[(iii)] $$\int_{B_{g(t_i)}(x_0,2R_i)}\left| \nabla^{g(t_i),2}u_{i,0}-\frac12 g(t_i)\right|^2_{g(t_i)} \,d\mu_{g(t_i)}\leq C(n,v_0)\e_i  $$
    \item[(iv)]
    $$|\nabla^{g(t_i)}u_{i,0}|_{g(t_i)}(\cdot)\leq C_n(d_{g(t_i)}(x_0,\cdot)+1) \text{  on  } B_{g(t_i)}(x_0,R_i).$$
 
\end{enumerate}
\end{lma}
\begin{proof}
The result is achieved by applying  \cite[Theorem 6.3]{CJN21} to appropriately rescaled metrics $g_R(t):=R^{-2}g(t)$. Indeed, by our assumption of the pointed Gromov-Hausdorff convergence to $C(X)$, for any $R>0$ and $\delta>0$,
$$\left(B_{g_R(t)}(x_0,\delta^{-1}), d_{g_R(t)} \right) = \left(B_{g(t)}(x_0,R\delta^{-1}), R^{-1} d_{g(t)} \right)$$ is $(0,\delta^2)$-symmetric for sufficiently small $t>0$. So, for any positive sequence $\varepsilon_i \downarrow 0$, 
   one may argue by \cite[Theorem 4.22 and Theorem 6.3]{CJN21} as in the proof of \cite[Proposition 3.1]{DSS2024} to see that for all sufficiently small positive $t$,
   there exists a smooth function $h: B_{g_R(t)}(x_0, 2)\to \R$ such that
   \begin{enumerate}
       \item[(i)'] $\Delta_{g_R(t)}h=\frac{n}{2}$;
       \item[(ii)'] $\sup_{B_{g_R(t)}(x_0,2)}\left|h-\frac{1}{4}d_{g_R(t)}^2(x_0,\cdot)\right|\leq \e_i^2$;
       \item[(iii)'] $\aint_{B_{g_R(t)}(x_0, 2)}\left|\nabla^{g_R(t),2}h-\frac{g_R(t)}{2}\right|^2\,d\mu_{g_R(t)}\leq C(n,v_0)\varepsilon_i^2$.
   \end{enumerate}
   Setting $u = R^2 h$, we can view $u$ as a function on $B_{g(t)}(x_0, 2R)=B_{g_R(t)}(x_0, 2)$, and using the volume comparison, rewrite the above estimates as
   \begin{enumerate}
       \item[(i)] $\Delta_{g(t)}u=\frac{n}{2}$;
       \item[(ii)] $\sup_{B_{g(t)}(x_0, 2R)}\left|u-\frac{1}{4}d_{g(t)}^2(x_0,\cdot)\right|\leq \e_i^2 R^2$;
       \item[(iii)] 
       \begin{eqnarray*}
           \int_{B_{g(t)}(x_0, 2R)}\left|\nabla^{g(t),2}u-\frac{g(t)}{2}\right|^2\,d\mu_{g(t)}&\leq& C(n,v)\varepsilon_i^2 \text{Vol}\left(B_{g(t)}(x_0, 2R)\right)\\
           &\leq& C'(n,v_0)\varepsilon_i^2 R^n.
       \end{eqnarray*}
   \end{enumerate}
 Since $R>0$ is arbitrary, we set $R_i:=\left(1/\varepsilon_i\right)^{\tfrac{1}{n}}\to +\infty$ and choose sufficiently small $t_i$ such that the above estimates hold. The resulting $u_{i,0}:=u$ satisfies (i), (ii) and (iii). Property (iv) follows by applying the interior gradient estimate to (i) and using the non-negative Ricci curvature and (ii).
\end{proof}
\begin{rem}
Although in the previous lemma we are assuming strong non-negativity of the curvature, the result still holds under the weaker assumption that $0\leq \Ric(g(t))\leq \a t^{-1}$ for some $\a>0$.
\end{rem}

\medskip

We want to use the $u_{i,0}$ obtained from Lemma~\ref{lma:appro-initial} to construct a global heat flow on $M\times (0,+\infty)$ with weak initial data $\frac14 d_0^2(\cdot, x_0)$. In order to do so, let $t_i \downarrow 0$ be the sequence of times obtained from Lemma~\ref{lma:appro-initial}, and let $\phi$ be a smooth non-increasing function on $[0,+\infty)$ which vanishes outside $[0,2]$ and is equal to $1$ on $[0,1]$. Then, we define a sequence of approximations $v_i$ to our desired solution via the formula
\begin{equation}\label{intf}
    v_i(x,t):=\int_M K(x,t;y,t_i)\cdot  \phi\left(\frac{d_{g(t_i)}(x_0,y)}{R_i}\right)u_{i,0}(y)\,d\mu_{g(t_i)}(y),
\end{equation}
where $K(x,t;y,s)$ denotes the heat kernel of the Ricci flow on $M\times (0,+\infty)$. Equivalently, $v_i$ is the minimal solution to the heat equation $\heat v_i = 0$ on $M\times [t_i,+\infty)$ with initial data $\phi \cdot u_{i,0}$ at $t=t_i$.

In order to extract our desired limiting solution, we require uniform bounds on the sequence $v_i$ locally in spacetime.
\begin{prop}\label{prop:vinf}
There exists $C_0(n,\a,v_0)>0$ such that the following holds. For any $r>1$, there exists $N(r) \in \mathbb{N}$ such that
{\begin{equation}\label{quab}
    |v_i(x,t)|\leq C_0 r^2, \quad \forall i > N, \quad \forall (x,t)\in B_{d_0}(x_0,r)\times [r^{-2}, r^2].
\end{equation}}
In particular, after passing to a subsequence, $v_i\to 
v_\infty$ in $C^\infty_{loc}\left(M\times (0,+\infty) \right)$ as $i\to+\infty$ for some $v_\infty$ on $M\times (0,+\infty)$ satisfying $\heat v_\infty=0$.
\end{prop}
\begin{proof}
Since $g(t)$ is smooth with $|\Rm(g(t))|\leq \a t^{-1}$ for $t>0$, once we have shown the $L^\infty_{loc}$ estimate \eqref{quab}, the existence of a smooth subsequential limit follows from standard parabolic theory. Moreover, by Lemma~\ref{lma:appro-initial} and \eqref{intf}, we see that $v_i\geq -\varepsilon_i \uparrow 0$ on $M\times [t_i,+\infty)$, with $\e_i,t_i\downarrow 0$. Therefore, when proving \eqref{quab}, it suffices to estimate the $v_i$ from above.

For any large $r>1$, fix $N\in \mathbb{N}$ so that $t_i<\frac12 r^{-2}$ for any $i > N$. Then, for any $(x,t)\in B_{d_0}(x_0,r)\times [r^{-2},r^2]$, since the integral of the heat kernel over $M$ is bounded by $1$, 
 \begin{eqnarray*}
      v_i(x,t)&:=&\int_M K(x,t;y,t_i)\cdot  \phi\left(\frac{d_{g(t_i)}(x_0,y)}{R_i}\right)u_{i,0}(y)\,d\mu_{g(t_i)}(y)\\
      &\leq&\int_{B_{g(t_i)(x_0,2R_i)}} K(x,t;y,t_i)\cdot u_{i,0}(y)\,d\mu_{g(t_i)}(y)\\
      &\leq&\int_{M} K(x,t;y,t_i)\cdot \frac{1}{4}d_{g(t_i)}^2(x_0,y)\,d\mu_{g(t_i)}(y)+\varepsilon_i.
 \end{eqnarray*}

To estimate the above integral, we first observe that by the triangle inequality and Proposition~\ref{prop:dist}, for any $x\in B_{d_0}(x_0,r)\Subset M$ and $y\in M$,
\[
d_{g(t_i)}^2(x_0,y)\leq 2d_{g(t_i)}^2(x_0,x) + 2d_{g(t_i)}^2(x,y)\leq 2r^2 + 2d_{g(t_i)}^2(x,y),
\]
so that 
 \begin{eqnarray*}
 &&\int_{M} K(x,t;y,t_i)\cdot \frac{1}{4}d_{g(t_i)}^2(x_0,y)\,d\mu_{g(t_i)}(y)\\
 &\leq&\frac12 \int_{M} K(x,t;y,t_i)\cdot d_{g(t_i)}^2(x, y)\,d\mu_{g(t_i)}(y)+2r^2\\
 &\leq&\int_0^\infty\frac{C_0'}{t^{n/2}}\text{  exp}\left(-\frac{\tau^2}{C_0 t}\right)\tau^{n+1} \,d\tau+2r^2\\
 &\leq& C_1 (r^2+t),
  \end{eqnarray*}
where we have used the heat kernel estimate from Theorem~\ref{thm:Ricciflow-PIC1}, the co-area formula, and the volume comparison for $g(t_i)$. This completes the proof.
\end{proof}

Given the construction of $v_\infty$ from Proposition~\ref{prop:vinf}, in order to obtain more reasonable convergence to $\frac{1}{4}d_0^2(\cdot,x_0)$ as $t \downarrow 0$, we require higher order estimates on our sequence $v_i$. Such estimates will also enable the use of a localised maximum principle in the proof of the main theorem \ref{thm:main-1}. 
\begin{lma}\label{lma:rough-bdd} 
For any $k\in \mathbb{N}$, there exist $C(n,k,\a)>0$ such that the following holds. For all $r> 1$ there exists $N(r,k) \in \mathbb{N}$ such that
\begin{equation}\label{Shivi}
|\nabla^{k, g_i} v_i|^2(x,t) \leq \frac{C(n,k,\a)}{(t-t_i)^{k-1}}r^{2}, \quad \forall i > N, \quad \forall (x,t) \in B_{g(t_i)}(x_0, r)\times (t_i,t_i+r^2].
\end{equation}
In particular
\begin{equation}\label{init}
\left|v_\infty(x,t)-\frac14 d_0(x,x_0)^2\right|\leq C(n,\a)r\sqrt{t}, \quad \forall t \in (0,\infty),
\end{equation}
where $v_{\infty}$ is the solution constructed in Proposition~\ref{prop:vinf}, so that $v_\infty(\cdot,t)\to \frac14 d_0(\cdot,x_0)^2$  as $t\downarrow 0$ in $C^\b_{loc}(M)$ (with respect to local structure)  for some $\b\in (0,1)$. 
\end{lma}
\begin{proof}
Recall that, thanks to our choice of cutoff function, $v_i$ is a solution to the heat equation with $v_i(\cdot,t_i)=u_{i,0}(\cdot)$ on $B_{g(t_i)}(x_0,R_i)$. Fix $r>1$ and consider the parabolic rescalings 
$$\tilde g_i(t)=r^{-2}g(t_i+r^2t), \quad \tilde v_i(\cdot,t)=r^{-2}v_i(\cdot,t_i+r^2 t).$$
Note that $\tilde v_i$ is a solution to the heat equation $\left(\frac{\partial}{\partial t} - \Delta_{\tilde g_i(t)}\right) \tilde v_i = 0$ with initial data $\tilde v_i(\cdot,0) = r^{-2}u_{i,0}(\cdot)$ on $B_{\tilde g_i(0)}(x_0,R_i r^{-1})$, and $|\Rm(\tilde g_i(t))|\leq \a t^{-1}$ on $M\times (0,1]$. Moreover, it follows from Shi's estimate that for all $k\in \mathbb{N}$, there exists $C(n,k,\a)>0$ such that for all $(x,t)\in M\times (0,1]$, $$|\nabla^{\tilde g_i(t),k}\Rm(\tilde g_i(t))|_{\tilde g_i(t)}\leq C(n,k,\a)t^{-(k/2+1)}.$$

Thanks to the local gradient estimates for $u_{i,0}$ in Lemma \ref{lma:appro-initial}, we have uniform gradient bounds for the $\tilde v_i$ on $B_{\tilde g_i(0)}(x_0,1)\times [0,1]$. Using the upper bound of the curvature and its derivatives, we may employ a standard argument for Ricci flows which uses the Bernstein-Shi trick and a Perelman type cutoff function to establish that 
\begin{equation}\label{stdshi}
|\nabla^{k,\tilde g_i} \tilde v_i|^2 \leq \frac{C(n,k,\a)}{t^{k-1}},
\end{equation}
for all $k\in \mathbb{N}$ on $B_{\tilde g_i(0)}(x_0,\frac12)\times (0,1]$. This is equivalent to \eqref{Shivi} under our rescaling. Furthermore, as 
\begin{equation}\label{eqn:local-holder-time}
|\partial_t v_i|=|\Delta_{g(t)}v_i| \leq C(n,\a)r (t-t_i)^{-1/2} , \quad \forall t \in (t_i,t_i+r^2],
\end{equation}
by integrating in time we obtain \eqref{init}.

It remains to show that $v_i\to v_\infty$ in $C^\b_{loc}(M\times [0,+\infty))$ for some $\b \in (0,1)$. Fix a smooth reference metric $h=g(1)$. By the gradient estimate \eqref{Shivi} with $k=1$, and the fact that $g(t)\leq g(t_i)$ for $t\geq t_i$ from $\Ric \geq 0$, we have 
\begin{equation}\label{eqn:local-holder}
    \left| v_i(x,t)-v_i(y,t)\right|\leq C(n,\a)r\cdot  d_{g(t_i)}(x,y),
\end{equation}
for all $x,y\in B_{g(t_i)}(x_0,r)$ and $t\in  [t_i,t_i+r^2]$. It follows from Proposition~\ref{prop:dist} that there exists $\b>0$ such that for any $r>1$, $v_i$ is uniformly $\b$-H\"older continuous with respect to the metric $h$ on $B_{h}(x_0,r)\times [t_i,t_i+r^2]$ as $i\to+\infty$. The result follows. 
\end{proof}
\begin{rem}
In fact, it follows from \eqref{eqn:local-holder-time} and \eqref{eqn:local-holder} that there exists $\b(n,\a)>0$ so that $$v_\infty\in C^\infty_{loc}\left(M\times (0,+\infty)\right)\cap C^{\b,\b/2}_{loc}\left(M\times [0,+\infty) \right),$$
where the parabolic H\"older regularity is defined using local charts on $M$ (or equivalently via a smooth metric $g(t_0)$) instead of using the rough distance metric $d_0$ on $M$, so that the regularity is with respect to the smooth structure of $M$. 
\end{rem}

\section{Ricci flows out of metric cones}
Now we are ready to prove Theorem~\ref{thm:main-1}. We begin with the Riemannian case.  More precisely, we show that if a Ricci flow coming out of a metric cone has non-negative $1$-isotropic curvature and Euclidean volume growth, then it must be an expanding Ricci soliton.
\begin{thm}\label{thm:Rie-case} 
Suppose $(M^n,g(t))$ is a complete non-compact Ricci flow on $M\times (0,+\infty)$ such that for some $v_0,\a>0$, the following properties hold for all $t>0$
\begin{enumerate}
    \item[(a)] $\mathrm{Rm}(g(t))\in \mathrm{C}_{PIC1}$;
    \item[(b)] $|\Rm(g(t))|\leq \a t^{-1}$;
    \item[(c)] $\mathrm{AVR}(g(t))\geq v_0>0$.
\end{enumerate}
Suppose further that $(M,d_0,x_0)$ is isometric to a metric cone $\left( C(X),d_c,o\right)$, where $d_0$ is the well-defined metric on $M$ given by the pointwise limit of $d_{g(t)}$ as $t \downarrow 0$ (see Proposition \ref{prop:dist}). Then there exists a smooth function $u$ on $M\times (0,+\infty)$ such that 
\begin{enumerate}
    \item[(i)] $2t\,\Ric(g(t))+g(t)-2\nabla^2 u=0$ on $M\times (0,+\infty)$;
    \item[(ii)] $u(\cdot,t)\to \frac14 d_0(x_0,\cdot)^2$ in $C^\b_{loc}(M)$ for some $\b\in (0,1)$ as $t \downarrow 0$.
\end{enumerate}
\end{thm}

\begin{rem}
As in the proof of Theorem~\ref{thm:Ricciflow-PIC1}, if the Ricci flow in Theorem~\ref{thm:Rie-case} is only known to exist for a short-time, then it can be extended to all $t>0$ with the same properties. 
\end{rem}

\begin{proof}
Letting $u:=v_\infty$ be the function obtained in Proposition~\ref{prop:vinf}, (ii) follows immediately from Lemma~\ref{lma:rough-bdd}. In order to prove (i), it is more convenient to translate our sequence of approximations in time by $t_i$, i.e. replace $g_i(t)$ by $g(t_i+t)$ and $v_i(t)$ by $v_i(t_i+t)$. We note that since $t_i\downarrow 0$, these translated approximations still converge locally smoothly to $v_\infty$, and thus it suffices to estimate the translated tensors
$$S_i(t) := -2t\Ric(g_i(t))-g_i(t)+2\nabla^{2,g_i(t)} v_i(t).$$ 

We let $\varphi_i \geq 0$ be the negative part of the lowest eigenvalue of $S_i\owedge g_i$ with respect to the cone of non-negative curvature operators as in \eqref{defn:min-max-S2}. If $G_i(x,t;y,s)$ denotes the conjugate heat kernel for the Ricci flow $g_i(t)$, define the comparison function 
$$
\psi_i(x,t) := \int_{M} G_i(x,t;y,0)\, \phi\left( \frac{d_{g_i(0)}(x_0,y)}{R_i}\right)\varphi_i(y,0)\,d\mu_{g_i(0)}(y) \geq 0,$$
where $\phi$ is a smooth non-increasing function on $[0,+\infty]$ which is identically $1$ on $[0,\tfrac{1}{2}]$ and vanishes outside of $[0,1]$. Thus, $\psi_i(\cdot,0) = \varphi_i(\cdot,0)$ on $B_{g_i(0)}(x_0, R_i/2)$, but more importantly, Lemma~\ref{lma:evo-S2} implies that $\Psi_i := \varphi_i-\psi_i$ satisfies 
\begin{equation}
    \left(\frac{\partial}{\partial t}-\Delta_{g_i(t)}\right) \Psi_i\leq \mathcal{R}_i \cdot \Psi_i,
\end{equation}
whenever $\Psi_i \geq 0$ in the barrier sense. For convenience, we now work on the re-scaled picture. For each sufficiently large $r>1$, we consider 
$$\tilde g_i(t)=r^{-2}g_i(r^2t), \quad \tilde v_i(\cdot,t)=r^{-2}v_i(\cdot,r^2 t), \quad \tilde S_i(t)=r^{-2}S_i(r^2t),$$
and the corresponding $\tilde\varphi_i=\varphi_i(\cdot, r^2t)$, $\tilde\psi_i=\psi_i(\cdot, r^2t)$ and $\tilde \Psi_i=\Psi_i(\cdot, r^2t)$.  It follows from Lemma~\ref{lma:rough-bdd} that
$$\tilde\Psi_i \leq \tilde\varphi_i\leq C(n,\a)t^{-1/2} \textnormal{   on   } B_{\tilde g_i(0)}(x_0,1) \times (0,1],$$
and therefore, we may apply Theorem~\ref{thm:Local-MP} to $\tilde \Psi_i$ to conclude that there exists $\hat T(n,\a) \in (0,1]$ such that $\tilde\Psi_i(x,t)\leq t$ on $B_{\tilde g_i(0)}(x_0,1/2)\times (0,\hat T]$. Rescaling back, we have shown that
\begin{equation}\label{eqn:Psi}
    \Psi_i(x,t)\leq r^{-2}t, \quad \forall (x,t)\in B_{g_i(0)}(x_0, r/2)\times (0,\hat T r^2].
\end{equation}
Next, we estimate $\psi_i$ from above. Using Lemma~\ref{lma:appro-initial} part (iii) and the fact that 
$$\varphi_i(\cdot,0) \leq 2 \cdot |S_i(0)|_{g_i(0)} = 4 \cdot \left| \nabla^{2,g_i(0)}v_i - g_i(0)/2\right|_{g_i(0)},$$ 
we see that
\begin{equation*}
\begin{split}
\psi_i(x,t)
    &\leq \left(\int_{B_{g_i(0)}(x_0,R_i)} \varphi_i^2\,d\mu_{g_i(0)}\right)^{1/2} \left(\int_{B_{g_i(0)}(x_0,R_i)} G_i^2(x,t;y,0)\,d\mu_{g_i(0)}(y) \right)^{1/2}\\
    &\leq  C(n,v_0) \cdot \e_i^{\frac12} \cdot \left(\int_{B_{g_i(0)}(x_0,2R_i)} G_i^2(x,t;y,0)\,d\mu_{g_i(0)}(y) \right)^{1/2}.
\end{split}
\end{equation*}
By the heat kernel estimate from Theorem~\ref{thm:Ricciflow-PIC1}, the co-area formula and volume comparison, for all $t>0$ we have
\begin{equation*}
    \begin{split}
        &\int_{B_{g_i(0)}(x_0,2R_i)} G_i^2(x,t;y,0) \,d\mu_{g_i(0)}(y) \\
        &\leq \frac{C_0(n,v_0,\a)}{t^{n/2}}\int_{M} G_i(x,t;y,0)\,d\mu_{g_i(0)}(y)\\
        &=\frac{C_0'}{t^{n/2}}\int_{0}^\infty \frac{r^{n-1}}{t^{n/2}}\exp\left( -\frac{r^2}{C_0 t}\right) dr \leq \frac{C_1(n,v,\a)}{t^{n/2}},
    \end{split}
\end{equation*}
and therefore
\begin{equation}\label{eqn:psi}
\psi_i(x,t) \leq C(n,v_0,\a) \cdot \e_i^{\frac12} \cdot t^{-n/4}.
\end{equation}

We combine  \eqref{eqn:Psi} and \eqref{eqn:psi}
to conclude that, for all $(x,t)\in M\times (0,+\infty)$, the inequality
\begin{equation*}
    \varphi_i(x,t) = \Psi_i(x,t) + \psi_i(x,t) \leq \frac{C(n,\a,v_0)}{t^{n/4}} \e_i^{\frac12} +r^{-2}t
\end{equation*}
holds for sufficiently large $i$  and $r>1$. Since $v_i\to v_\infty$ and $g_i\to g$ locally uniformly smoothly on $M\times (0,+\infty)$,  by letting $i\to +\infty$, followed by $r\to +\infty$, we conclude that $S\owedge g\geq 0$ on $M\times (0,+\infty)$. By repeating the argument with the positive part of the largest eigenvalue of $S_i\owedge g_i$ with respect to the cone of non-positive curvature operators, we deduce that $S\owedge g\equiv 0$ and hence $S\equiv 0$ on $M\times (0,+\infty)$ by Lemma~\ref{lma:linear-alge}. This completes the proof. 
\end{proof}

Theorem~\ref{thm:main-1} now follows from Theorem~\ref{thm:Rie-case} and Theorem~\ref{thm:Ricciflow-PIC1}.
\begin{proof}[Proof of Theorem~\ref{thm:main-1}]
If $(C(X),d_c,o)$ is a metric cone at infinity of $(M,g_0)$, there exists $x_0 \in M$ and $R_i \to +\infty$ such that $(C(X),d_c,o)$ is the pointed Gromov-Hausdorff limit of $(M,R_i^{-2}g_0,x_0)$. By applying Theorem~\ref{thm:Ricciflow-PIC1} to each $g_{i,0}=R_i^{-2}g_0$, we obtain a sequence of immortal Ricci flows $(M,g_i(t),x_0),t\in [0,+\infty)$. The existence of $(M_\infty,g(t),x_\infty)$ with properties (a)-(d) follows from Hamilton's compactness theorem \cite{Hamilton-compact} and Proposition~\ref{prop:dist}. The existence of $u$ with properties (e)-(f) follows from Theorem~\ref{thm:Rie-case}.
\end{proof}

\medskip
As an application, we give an alternative proof to \cite[Theorem 1.3]{DSS2024}.
\begin{cor}\label{cor:pinching}
Suppose $(M^n,g_0)$ is a complete non-compact manifold with dimension $n\geq 3$ and of Euclidean volume growth such that it is PIC1 pinched in the following sense:
\begin{equation}\label{eqn:PIC1 pinched}
\mathrm{Rm}(g_0)-\e_0\cdot \mathcal{R} g_0\owedge g_0\in \mathrm{C}_{PIC1},
\end{equation}
for some small $\e_0 \in (0,\frac{1}{2n(n-1)})$. Then $(M,g_0)$ is isometric to flat Euclidean space.
\end{cor}
\begin{proof}
By our choice of $\e_0$, \eqref{eqn:PIC1 pinched} implies that $\mathcal{R} \geq 0$ and $\mathrm{Rm}(g_0) \in \mathrm{C}_{PIC1}$. By Gromov compactness and the celebrated work of Cheeger-Colding \cite{CheegerColding}, there exists $R_i\to+\infty$ and a metric cone $C(X)$ such that $(M_i,g_i,x_i)=(M,R_i^{-2}g_0,x_0)$ converges to $C(X)$ in the pointed Gromov-Hausdorff sense as $i\to+\infty$. We will follow the proof of Theorem~\ref{thm:main-1}. For the sake of convenience, we instead apply the existence theory in \cite{LeeToppingPIC1}. Since \eqref{eqn:PIC1 pinched} is scaling invariant, it follows from \cite[Theorem 1.3]{LeeToppingPIC1} that there exists a sequence of immortal Ricci flows $g_i(t)$ on $M$ with $g_i(0)=R_i^{-2}g_0$ that is PIC1 pinched with a slightly different $\e_0$ which is uniform in $i\to+\infty$ and $t>0$. Moreover, it follows from the proof of \cite[Theorem 7]{Yok} that $\mathrm{AVR}(g_i(t))=\mathrm{AVR}(g_0)$ for all $i\in \mathbb{N}$ and $t\geq 0$. We now follow the proof of Theorem~\ref{thm:main-1} to obtain a smooth manifold $M_\infty$ and an immortal Ricci flow $g_\infty(t),t\in (0,+\infty)$ on $M_\infty$ which is coming out of $C(X)$ in the pointed Gromov-Hausdorff sense. Moreover $\mathrm{AVR}(g_\infty(t))=\mathrm{AVR}(g_0)$ for $t>0$ using Colding's volume continuity \cite{ColdingVol}, see the proof of \cite[Theorem 1.2]{SchulzeSimon2013}. It follows from Theorem~\ref{thm:Rie-case} that $(N,g,\nabla f)=\left(M_\infty,g_\infty(1),\nabla u(1)\right)$ is an expanding gradient Ricci soliton  which is also PIC1 pinched.


Since $(N,g,\nabla f)$ is Ricci pinched, the scalar curvature decays exponentially at spatial infinity (see for example the proof of \cite[Proposition 1.8]{Deruelle2017}). Thanks to $\mathrm{Rm}(g)\in \mathrm{C}_{PIC1}$, by \cite[Lemma A.2]{LeeToppingPIC1} we can improve this to the full curvature tensor decaying exponentially at spatial infinity, and thus $(N,g)$ is isometric to flat Euclidean space by \cite[Theorem 1.3 \& 1.4]{ChanLee2023}. In particular, $\mathrm{AVR}(N,g)=\mathrm{AVR}(M,g_0)=1$. The flatness of $(M,g_0)$ follows from volume comparison. 
\end{proof}

\medskip

In the K\"ahler case, it is also interesting to see how the complex structure and metric structure are related. We have the following partial answer under the curvature assumption $\mathrm{OB},\Ric\geq 0$. This condition is weaker than non-negative $1$-isotropic curvature and is much more natural in the context of complex geometry.
\begin{thm}\label{thm:Kahler-OB-Ric} 
Suppose $M$ is a non-compact complex manifold and $g(t)$ is a complete non-compact \KR flow on $M\times (0,+\infty)$ such that for some $v_0,\a>0$, the following conditions hold for all $t>0$
\begin{enumerate}
    \item[(a)] $\Ric(g(t)),\mathrm{OB}(g(t))\geq 0$;
    \item[(b)] $|\Rm(g(t))|\leq \a t^{-1}$;
    \item[(c)] $\mathrm{AVR}(g(t))\ge v_0>0$.
\end{enumerate}
Suppose further that $(M,d_0,x_0)$ is isometric to a metric cone $\left( C(X),d_c,o\right)$, where $d_0$ is the well-defined metric on $M$ given by the pointwise limit of $d_{g(t)}$ as $t \downarrow 0$ (see Proposition \ref{prop:dist}). Then there exists a smooth function $u$ on $M\times (0,+\infty)$ such that 
\begin{enumerate}
    \item[(i)] $t\,\Ric(\omega(t))+\omega(t)-\ddb u=0$ on $M\times (0,+\infty)$;
     \item[(ii)] $u(\cdot,t)\to \frac12 d_0(\cdot,x_0)^2$ in $C^\b_{loc}(M)$ for some $\b\in (0,1)$ as $t\downarrow 0$;
    \item[(iii)] $\omega(t)-\frac12  \ddb d_0(\cdot,x_0)^2 \to 0$ in the sense of currents as $t\downarrow 0$.
\end{enumerate}
In particular, $\omega_\infty=\frac12  \ddb d_0(\cdot,o)^2$ defines a weak \K metric on $M_\infty$. Furthermore, if condition (a) is strengthened to non-negative $1$-isotropic curvature, then $M^n$ is biholomorphic to $\mathbb{C}^n$ and $g(t)$ is an expanding \KR soliton. 
\end{thm}
\begin{proof}
The existence of a function $u$ with property (ii) follows from Proposition~\ref{prop:vinf} and Lemma~\ref{lma:rough-bdd}. The proof of (i) is identical to that of (i) in Theorem~\ref{thm:Rie-case} by using Lemma~\ref{lma:evo-S-Kahler} instead of Lemma~\ref{lma:evo-S2}.

It remains to verify (iii). The existence of $\lim_{t\to 0}\omega(t)$ as a current has been shown by Lott \cite{LottDuke}. In order to identify it, we will use an idea from \cite[Proposition 6.1]{LottDuke}. By (i) and (ii), it suffices to show that $t\cdot \Ric(\omega(t))\to 0$ as $t\downarrow 0$ in the sense of currents. 

We might compare $\Ric(\omega(t))$ as $t\downarrow 0$ with $\Ric(\omega(1))$ in the following way:
\begin{equation}
  t\cdot   \Ric(\omega(t))=t\cdot \Ric(\omega(1))+t\ddb \log \frac{\det g(1)}{\det g},
\end{equation}
where $\omega(1)$ is of bounded curvature on $M$. Since $0\leq \mathcal{R}(g(t))\leq \a t^{-1}$ for some $\a>0$, the Ricci flow equation implies 
\begin{equation}
    0\leq t \log \left(\frac{\det g(t)}{\det g(1)}\right)\leq t\log t^{-\a}\to 0
\end{equation}
as $t\downarrow 0$.  Thus $\omega_\infty:=\ddb \frac12 d_0^2(x_0,\cdot)$ defines a weak \K metric on $M$.

Finally if the condition (a) is strengthened to non-negative $1$-isotropic curvature, it follows from Theorem~\ref{thm:Rie-case} that $g(t)$ is \K and an expanding Ricci soliton with $\Ric\geq 0$. Therefore, it is an expanding \KR soliton. The biholomorphism follows from \cite{ChauTamMRL}.
\end{proof}

Corollary~\ref{cor:kahler} follows from the \K case in Theorem~\ref{thm:Ricciflow-PIC1}, Cheeger-Colding theory \cite{CheegerColding} and Theorem~\ref{thm:Kahler-OB-Ric}. In the compact case, it is already known by the work \cite{ChenX,wilking} that compact \K manifold with $\mathrm{OB}>0$ are biholomorphic to $\mathbb{CP}^n$. It will be interesting to see if $M$ in Corollary~\ref{cor:kahler} is biholomorphic to $\mathbb{C}^n$ under only $\mathrm{OB},\Ric\geq 0$ and Euclidean volume growth.

\end{document}